\documentclass[a4paper,10pt]{article}
\usepackage{amsmath,amsthm,xypic,
amssymb,amsfonts,
pdfsync,stmaryrd,mathrsfs,yfonts}
\usepackage[capitalise]{cleveref}
\newtheorem{thm}{Theorem}[section]

\newtheorem{prop}[thm]{Proposition}
\newtheorem{cor}[thm]{Corollary}

\newtheorem{lem}[thm]{Lemma}
\theoremstyle{definition}
\newtheorem{defi}[thm]{Definition}
\newtheorem{defis}[thm]{Definitions}
\newtheorem{conj}[thm]{Conjecture}

\newtheorem{rem}[thm]{Remark}

\theoremstyle{plain}

\renewcommand{\phi}{\varphi}

\newcommand{\alt}{\mathop{\mathrm{Alt}}}

\newcommand{\sym}{\mathop{\mathrm{Sym}}}

\newcommand{\id}{\mathop{\mathrm{id}}}

\newcommand\iso{\xrightarrow{
   \,\smash{\raisebox{-0.3ex}{\ensuremath{\scriptstyle\sim}}}\,}}
\newcommand{\an}{\mathop{\mathrm{an}}}

\author{A.-H. Nokhodkar}
\date{}
\title
{Separable extensions of orthogonal involutions in characteristic two}

\begin{document}
\maketitle
\vspace{-.9cm}
\begin{center}
{\footnotesize
{\em
\baselineskip.6cm
 Department of Pure Mathematics, Faculty of Science, University of Kashan,

 P.~O. Box 87317-51167, Kashan, Iran.

    a.nokhodkar@kashanu.ac.ir}
}
\end{center}
\vspace{.01cm}
\begin{abstract}
It is shown that an anisotropic orthogonal involution in characteristic two is totally decomposable if it is totally decomposable over a separable extension  of the ground field.
In particular, this settles a characteristic two analogue of a conjecture formulated by Bayer-Fluckiger et al.\\
\\
\noindent
\emph{Mathematics Subject Classification:} 16W10, 16K20, 11E04. \\
\end{abstract}

\section{Introduction}
The theory of involutions on central simple algebras is a generalization of the theory of symmetric bilinear forms.
Given the fundamental role which Pfister forms play in the modern theory of bilinear forms, it is natural to ask whether there is a type of algebras with involution naturally generalising these forms.
Since the adjoint involution of a Pfister form is totally decomposable, the class of totally decomposable involutions is a good candidate for this generalisation.
Considering this fact and basic properties of Pfister forms, the following conjecture was formulated in \cite{bayer}:
\begin{conj}\label{conj}
Let $F$ be a field of characteristic not two and let $(A,\sigma)$ be an algebra with orthogonal involution of degree $2^n$ over $F$.
Then the following statements are equivalent:
\begin{itemize}
  \item [(1)] $(A,\sigma)$ is totally decomposable.
  \item [(2)] For every field extension $K/F$, $(A,\sigma)_K$ is either anisotropic or hyperbolic.
  \item [(3)] For every splitting field $K$ of $A$, $(A,\sigma)_K$ is adjoint to a Pfister form.
\end{itemize}
\end{conj}
The implications $(1)\Rightarrow (2)$, $(1)\Rightarrow (3)$ and $(2)\Leftrightarrow(3)$ are already known (see \cite[Theorem 1]{becher}, \cite[(1.1)]{kar} and \cite[(3.2)]{black}).
However, the implication $(2)\Rightarrow (1)$ or $(3)\Rightarrow (1)$ is still open in general, only known for
algebras of low index or degree or in certain special cases.

An analogous conjecture in characteristic two was formulated in \cite{dolphin3} with two differences:
(i) it was assumed that $\sigma$ is anisotropic;
(ii) the hyperbolicity condition in (2) was replaced by the metabolicity.
Note that the restriction to anisotropic involutions is reasonable as this restriction also must be made for the corresponding statement for bilinear Pfister forms in characteristic two (see \cite[(5.5)]{dolphin3}).
The implications $(1)\Rightarrow(2)$ and $(2)\Rightarrow(3)$ were proved in \cite{dolphin3}.
Also, a proof of $(3)\Rightarrow(1)$ is recently obtained in \cite{mep}.

According to \cite[(3.3)]{dolphin3}, a symmetric bilinear form over a field $F$ of characteristic two is similar to a Pfister form if it is similar to a Pfister form over some separable extension of $F$.
An analogue question for involutions may be considered as follows:
let $(A,\sigma)$ be a central simple algebra with involution over a field $F$ of characteristic two.
Given that $(A,\sigma)_K$ is totally decomposable for some separable extension $K/F$, does it imply that $(A,\sigma)$ is totally decomposable?
This question is in fact a generalization of the implication $(3)\Rightarrow(1)$ of \cref{conj} in characteristic two.

In this work, we study the above question for anisotropic orthogonal involutions.
Our approach is based on a totally singular quadratic form $q_\sigma$ associated to every orthogonal involution $\sigma$ in characteristic two.
This form was defined in \cite{me} and used to obtain some sufficient conditions for anisotropy of orthogonal involutions.
Although the form $q_\sigma$ is not functorial, it will be shown in \cref{sep} that if $K/F$ is a separable extension then $q_{\sigma_K}\simeq (q_\sigma)_K$.
We then prove in \cref{septd} that an $F$-algebra with anisotropic orthogonal involution $(A,\sigma)$ in characteristic two is totally decomposable if $(A,\sigma)_K$ is totally decomposable for some separable extension $K/F$
(note that this result is already known to be false if $\sigma$ is isotropic; see \cref{rem}).
As an application, we prove in \cref{mainn} that the three conditions of \cref{conj} (in characteristic two) are equivalent to:
\\
(4) $(A,\sigma)_K$ is totally decomposable for some separable extension $K/F$.

\section{Preliminaries}
In this work, all fields are implicitly supposed to be of characteristic two.\\
Let $V$ be a finite-dimensional vector space over a field $F$ and let $q$ be a quadratic form on $V$.
We say that $q$ is {\it isotropic} if $q(v)=0$ for some nonzero vector $v\in V$ and {\it anisotropic} otherwise.
The quadratic form $q$ is called {\it totally singular} if there exists a symmetric bilinear form $\mathfrak{b}$ on $V$ such that $q(v)=\mathfrak{b}(v,v)$ for every $v\in V$.
For a field extension $K/F$, the scalar extension of $q$ to $K$ is denoted by $q_K$.

Let $A$ be a central simple algebra over a field $F$.
An {\it involution} on $A$ is an anti-automorphism of order two of $A$.
If $\sigma$ is an involution on $A$ then the restriction $\sigma|_F$ is an automorphism which is either the identity map or of order two.
If $\sigma|_F=\id$ we say that $\sigma$ is of {\it the first kind}.
Otherwise, it is said to be of {\it the second kind}.
An involution of the first kind on $A$ is called {\it symplectic} if it becomes adjoint to an alternating bilinear form over a splitting field of $A$.
Otherwise, it is called {\it orthogonal}.
The sets of {\it symmetric} and {\it alternating} elements of an algebra with involution $(A,\sigma)$ are defined respectively as
\[\sym(A,\sigma)=\{x\in A\mid \sigma(x)=x\}\quad {\rm and}\quad \alt(A,\sigma)=\{x-\sigma(x)\mid x\in A\}.\]
According to \cite[(2.6)]{knus}, an involution $\sigma$ of the first kind on $A$ is orthogonal if and only if $\alt(A,\sigma)\cap F=\{0\}$.
An involution $\sigma$ on $A$ is called {\it isotropic} if there exists a nonzero element $x\in A$ such that $\sigma(x)x=0$ and
{\it anisotropic} otherwise.

For a central simple algebra $A$ over a field $F$ the integer $\sqrt{\dim_FA}$ is called the {\it degree} of $A$ and is denoted by $\deg_FA$.
A central simple algebra of degree two is called a {\it quaternion algebra}.
An algebra with involution is called {\it totally decomposable} if it decomposes as tensor products of quaternion algebras with involution.

\section{Separable  extensions of the alternator form}\label{sec-alt}
Let $(A,\sigma)$ be a central simple algebra with orthogonal involution over a field $F$ and
set
\[S(A,\sigma)=\{x\in A\mid\sigma(x)x\in F\oplus\alt(A,\sigma)\}.\]
Define a map $q_\sigma:S(A,\sigma)\rightarrow F$ as follows:
if $x\in S(A,\sigma)$ then there exists a unique element $\alpha\in F$ satisfying $\sigma(x)x+\alpha\in\alt(A,\sigma)$;
set $q_\sigma(x)=\alpha$.
By \cite[(3.2) and (3.3)]{me}, the set $S(A,\sigma)$ is an $F$-subalgebra of $A$ and $q_\sigma$ is a totally singular quadratic form on $S(A,\sigma)$.
As in \cite{me}, we call $q_\sigma$ the {\it alternator form} of $(A,\sigma)$.
It is easily verified that if $K/F$ is a field extension then $S(A,\sigma)\otimes K\subseteq S(A_K,\sigma_K)$ (see the proof of \cite[(3.6)]{me}).
However, the definition of $q_\sigma$ is not functorial by \cite[(3.18)]{me}.
In this section we show that if $K/F$ is a separable extension then $q_{\sigma_K}\simeq(q_\sigma)_K$.

\begin{lem}\label{quadratic}
Let $(A,\sigma)$ be a central simple algebra with orthogonal involution over a field $F$.
If $K/F$ is a separable quadratic extension, then $S(A_K,\sigma_K)=S(A,\sigma)\otimes K$.
\end{lem}

\begin{proof}
We only need to prove $S(A_K,\sigma_K)\subseteq S(A,\sigma)\otimes K$.
Write $K=F(\eta)$ for some $\eta\in K$ with $\delta:=\eta^2+\eta\in F$.
Then every $x\in S(A_K,\sigma_K)$ can be written as $x=u\otimes1+v\otimes\eta$, where $u,v\in A$.
We have
\begin{align}\label{eq8}
\sigma_K(x)x&=(\sigma(u)\otimes1+\sigma(v)\otimes\eta)(u\otimes1+v\otimes\eta)\nonumber\\
&=(\sigma(u)u+\delta \sigma(v)v)\otimes1+(\sigma(u)v+\sigma(v)u+\sigma(v)v)\otimes\eta.
\end{align}
Write $q_{\sigma_K}(x)=a\otimes1+b\otimes\eta$ for some $a,b \in F$, so that $\sigma_K(x)x+a\otimes1+b\otimes\eta\in\alt(A_K,\sigma_K)$.
Since $\alt(A_K,\sigma_K)=\alt(A,\sigma)\otimes K$, using (\ref{eq8}) we get
\begin{align}
\sigma(u)u+\delta \sigma(v)v+a\in\alt(A,\sigma),\label{eq1}\\
\sigma(u)v+\sigma(v)u+\sigma(v)v+b\in\alt(A,\sigma).\label{eq2}
\end{align}
As $\sigma(u)v+\sigma(v)u\in\alt(A,\sigma)$,
 (\ref{eq2}) implies that $\sigma(v)v+b\in\alt(A,\sigma)$, i.e., $v\in S(A,\sigma)$ and $q_\sigma(v)=b$.
Using (\ref{eq1}) we get
\[\sigma(u)u+a+b\delta=(\sigma(u)u+\delta\sigma(v)v+a)+\delta(\sigma(v)v+b)\in\alt(A,\sigma).\]
Hence, $u\in S(A,\sigma)$ and $q_\sigma(u)=a+b\delta$.
It follows that $x\in S(A,\sigma)\otimes K$, proving the result.
\end{proof}

\begin{lem}\label{alt}
Let $(A,\sigma)$ be a central simple algebra with involution over a field $F$ and let $x=\sum_{i=1}^nx_i$, where $x_1,\cdots,x_n\in A$.
Then $\sigma(x)x+\sum_{i=1}^n\sigma(x_i)x_i\in\alt(A,\sigma)$.
\end{lem}

\begin{proof}
We have
\begin{align*}
\sigma(x)x=\textstyle\sum_{i=1}^{n}\sigma(x_i)\cdot\sum_{i=1}^nx_i=\sum_{i=1}^{n}\sigma(x_i)x_i+\sum_{i\neq j}\sigma(x_i)x_j.
\end{align*}
Hence,
\begin{align*}
\sigma(x)x+\textstyle\sum_{i=1}^n\sigma(x_i)x_i&=\textstyle\sum_{i\neq j}\sigma(x_i)x_j=\textstyle\sum_{i<j}\sigma(x_i)x_j+\sum_{i<j}\sigma(x_j)x_i\\
&=\textstyle\sum_{i<j}(\sigma(x_i)x_j-\sigma(\sigma(x_i)x_j))\in\alt(A,\sigma).\qedhere
\end{align*}
\end{proof}

\begin{lem}\label{odd}
Let $(A,\sigma)$ be a central simple algebra with orthogonal involution over a field $F$.
If $K/F$ is a field extension of odd degree, then $S(A_K,\sigma_K)=S(A,\sigma)\otimes K$.
\end{lem}

\begin{proof}
As already observed, it is enough to show that $S(A_K,\sigma_K)\subseteq S(A,\sigma)\otimes K$.
Write $K=F(\eta)$ for some $\eta\in K$ and let $[K:F]=2n+1$.
Then the set $\{1,\eta,\eta^2,\cdots,\eta^{2n}\}$ is a basis of $K$ over $F$.
Let $x\in S(A_K,\sigma_K)$ and set $\alpha=q_{\sigma_K}(x)\in K$.
Write $x=\sum_{i=0}^{2n}x_i\otimes\eta^i$, where $x_i\in A$ for $i=0,\cdots,2n$.
Since $\sigma_K(x)x+\alpha\in\alt(A_K,\sigma_K)$, \cref{alt} implies that
\[\textstyle\sum_{i=0}^{2n}\sigma_K(x_i\otimes\eta^i)\cdot(x_i\otimes\eta^i)+\alpha\in\alt(A_K,\sigma_K).\]
Hence,
\begin{equation}\label{eq3}
\textstyle\sum_{i=0}^{2n}\sigma(x_i)x_i\otimes\eta^{2i}+\alpha\in\alt(A_K,\sigma_K).
\end{equation}
Since $[K:F]$ is odd, we have $F(\eta^2)=K$, hence the set $\{1,\eta^2,\eta^4,\cdots,\eta^{4n}\}$ is a basis of $K$ over $F$.
Write $\alpha=\sum_{i=0}^{2n}\alpha_i\otimes\eta^{2i}$ for some $\alpha_i\in F$, $i=0,\cdots,2n$.
Then (\ref{eq3}) implies that
\[\textstyle\sum_{i=0}^{2n}(\sigma(x_i)x_i+\alpha_i)\otimes\eta^{2i}\in\alt(A_K,\sigma_K).\]
 It follows that $\sigma(x_i)x_i+\alpha_i\in\alt(A,\sigma)$, i.e., $x_i\in S(A,\sigma)$ for every $i$.
\end{proof}

Using \cref{quadratic}, \cref{odd} and basic Galois theory, we have the following result.
\begin{cor}\label{fin}
Let $(A,\sigma)$ be a central simple algebra with orthogonal involution over a field $F$.
If $K/F$ is a finite separable extension, then $S(A_K,\sigma_K)=S(A,\sigma)\otimes K$.
\end{cor}

\begin{thm}\label{sep}
Let $(A,\sigma)$ be a central simple algebra with orthogonal involution over a field $F$.
If $K/F$ is a separable extension, then $S(A_K,\sigma_K)=S(A,\sigma)\otimes K$ and $(S(A_K,\sigma_K),q_{\sigma_K})\simeq(S(A,\sigma),q_\sigma)_K$.
\end{thm}

\begin{proof}
We already know that $S(A,\sigma)\otimes K\subseteq S(A_K,\sigma_K)$.
Let $x\in S(A_K,\sigma_K)$ and set $\alpha=q_{\sigma_K}(x)\in K$.
As $K/F$ is separable, one can find a subfield $L$ of $K$ containing $F$ with $[L:F]<\infty$ for which $x\in A_L$ and $\alpha\in L$.
We have
\[\sigma_L(x)x+\alpha=\sigma_K(x)x+\alpha\in\alt(A_K,\sigma_K)\cap A_L=\alt(A_L,\sigma_L).\]
It follows that $x\in S(A_L,\sigma_L)$.
By \cref{fin} we have $x\in S(A,\sigma)\otimes L\subseteq S(A,\sigma)\otimes K$.
This proves that $S(A_K,\sigma_K)=S(A,\sigma)\otimes K$.
On the other hand, it is easily seen that $q_{\sigma_K}|_{S(A,\sigma)\otimes K}=(q_\sigma)_K$ (see the proof of \cite[(3.6)]{me}).
Hence, $q_{\sigma_K}=(q_\sigma)_K$, proving the result.
\end{proof}

\begin{defi}
An algebra with involution $(A,\sigma)$ (or the involution $\sigma$ itself) is called {\it direct} if there is no nonzero element $x\in A$ satisfying $\sigma(x)x\in\alt(A,\sigma)$.
\end{defi}
Direct involutions were introduced in \cite{dolphin2}.
Note that every direct involution is anisotropic.
Also, by \cite[(9.3)]{dolphin2} every direct involution stays direct over a separable extension of the ground field.
We present here an alternative proof of this fact.
Recall that any anisotropic totally singular quadratic form remains anisotropic over a separable extension.
This follows from the corresponding result for bilinear forms (see \cite[(10.2.1)]{kne}).
\begin{cor}
Let $K/F$ be a field extension of odd degree and let $(A,\sigma)$ be a central simple algebra with orthogonal involution over
$F$.
Then $\sigma$ is direct if and only if $\sigma_K$ is direct.
\end{cor}

\begin{proof}
According to \cite[(3.8)]{me}, $\sigma$ is direct if and only if $q_\sigma$ is anisotropic.
The conclusion therefore follows from \cref{sep} and the above result on separable extensions.
\end{proof}
\section{The main result}\label{sec-main}
In this section we study orthogonal Pfister involutions in characteristic two.
We first recall a definition from \cite{mn1}.
\begin{defi}
A unitary algebra $R$ over a field $F$ is called a {\it totally singular conic $F$-algebra} if $x^2\in F$ for every $x\in R$.
\end{defi}
In \cite{mn1}, it was shown that a central simple algebra of degree $2^n$ with involution of the first kind is totally decomposable if and only if there exists a $2^n$-dimensional totally singular conic $F$-algebra $\Phi\subseteq \sym(A,\sigma)$, which is generated as an $F$-algebra by $n$ elements and satisfies $C_A(\Phi)=\Phi$, where $C_A(\Phi)$ is the centralizer of $\Phi$ in $A$ (see \cite[(4.6)]{mn1}).
According to \cite[(5.10)]{mn1}, if $\sigma$ is orthogonal then $\Phi$ is uniquely determined, up to isomorphism.

The next result gives another criterion for an anisotropic orthogonal involution to  be totally decomposable.

\begin{prop}\label{td}
Let $(A,\sigma)$ be a central simple algebra of degree $2^n$ with anisotropic orthogonal involution over a field $F$.
Then $(A,\sigma)$ is totally decomposable if and only if $S(A,\sigma)\subseteq\sym(A,\sigma)$ and $\dim_FS(A,\sigma)=2^n$.
\end{prop}

\begin{proof}
The `only if' implication follows from \cite[(4.1)]{me}.
Conversely, suppose that $S(A,\sigma)\subseteq\sym(A,\sigma)$ be $2^n$-dimensional.
By \cite[(3.10)]{me}, $\sigma$ is direct.
Hence, \cite[(3.7)]{me} implies  that $S(A,\sigma)$ is a field satisfying $x^2\in F$ for every $x\in S(A,\sigma)$.
In particular, $S(A,\sigma)$ is a totally singular conic $F$-algebra.
Since $\dim_FS(A,\sigma)=2^n=\deg_FA$, $S(A,\sigma)$ is a maximal subfield of $A$, hence $C_A(S(A,\sigma))=S(A,\sigma)$.
Finally, note that $S(A,\sigma)$ is generated as an $F$-algebra by $n$ elements, because $S(A,\sigma)$ is a field.
Thus, $(A,\sigma)$ is totally decomposable by \cite[(4.6)]{mn1}.
\end{proof}
Recall that an algebra with involution $(A,\sigma)$ (or the involution $\sigma$ itself) over a field $F$ is called {\it metabolic} if there is an idempotent $e\in A$ with $\dim_FeA=\frac{1}{2}\dim_FA$ such that $\sigma(e)e=0$ (see \cite[\S A.1]{ber}).
In this case, we say that $e$ is a {\it metabolic} idempotent with respect to $\sigma$.

We now recall some definitions from \cite{dolphin2}.
\begin{defis}
Let $(A,\sigma)$ be a central simple algebra with involution over a field $F$.
An algebra with involution $(B,\tau)$ is called a {\it part} of $(A,\sigma)$ if there exist a symmetric idempotent $e\in A$ and an $F$-algebra isomorphism $f:eAe\iso B$ such that $f\circ\tau=\sigma\circ f$
(note that the $F$-algebra $eAe$ is central simple with unit element $e$).
In this case, we say that $e$ {\it defines} $(B,\tau)$.
The part of $(A,\sigma)$ defined by $1-e$ is called a {\it counterpart} of $(B,\tau)$ in $(A,\sigma)$.
A part $(B,\tau)$ of $(A,\sigma)$ is called the {\it anisotropic part} of $(A,\sigma)$, if $(B,\tau)$ is anisotropic and any counterpart of $(B,\tau)$ in $(A,\sigma)$ is metabolic.
We denote the anisotropic part of $(A,\sigma)$ by $(A,\sigma)_{\an}$.
Finally, the {\it direct part} of $(A,\sigma)$ is a part $(C,\rho)$ of $(A,\sigma)_{\an}$ which is direct and any of its counterparts in $(A,\sigma)_{\an}$ is symplectic.
\end{defis}
Note that by \cite[(4.5) and (7.5)]{dolphin2} the anisotropic and direct parts of $(A,\sigma)$ are uniquely determined, up to isomorphism.
\begin{thm}\label{septd}
Let $K/F$ be a separable field extension and let $(A,\sigma)$ be a central simple algebra of degree $2^n$ with anisotropic orthogonal involution over $F$.
Then $(A,\sigma)$ is totally decomposable if and only if $(A,\sigma)_K$ is totally decomposable.
\end{thm}

\begin{proof}
The `only if' implication is evident.
Conversely, suppose that $(A,\sigma)_K$ is totally decomposable.
Since $\sigma$ is anisotropic, \cite[(7.7)]{dolphin2} implies that $\sigma$ has nontrivial direct part.
By \cite[(9.3)]{dolphin2} the direct part of $\sigma_K$ is also nontrivial.
In particular, $\sigma_K$ is not metabolic.
It follows from \cite[(6.2)]{dolphin3} that $\sigma_K$ is anisotropic.
By \cref{td} we have
\[S(A_K,\sigma_K)\subseteq\sym(A_K,\sigma_K)\quad {\rm and} \quad \dim_KS(A_K,\sigma_K)=2^n.\]
We also have $S(A_K,\sigma_K)=S(A,\sigma)\otimes K$ by \cref{sep}, hence $\dim_FS(A,\sigma)=2^n$.
Finally, the equality $\sym(A_K,\sigma_K)=\sym(A,\sigma)\otimes K$ implies that $S(A,\sigma)\subseteq\sym(A,\sigma)$.
Hence, $(A,\sigma)$ is totally decomposable by \cref{td}.
\end{proof}

\begin{rem}\label{rem}
\cref{septd} is not true for isotropic involutions.
Indeed, as observed in \cite[(9.4)]{dolphin3}, there exist a field $F$ and an $F$-algebra with isotropic orthogonal involution $(A,\sigma)$ which is not totally decomposable, but $(A,\sigma)_K$ is totally decomposable for any splitting field $K$ of $A$.
Hence, in this case, \cref{septd} is false for every separable splitting field $K$ of $A$.
\end{rem}

The following result gives an alternative proof of \cite[(3.9)]{mep}.
\begin{thm}\label{mainn}
For an algebra with anisotropic orthogonal involution $(A,\sigma)$ of degree $2^n$ over a field $F$ the following statements are equivalent:
\begin{itemize}
  \item [(1)] $(A,\sigma)$ is totally decomposable.
  \item [(2)] For every field extension $K/F$, $(A,\sigma)_K$ is either anisotropic or metabolic.
  \item [(3)] For every splitting field $K$ of $A$, $(A,\sigma)_K$ is adjoint to a Pfister form.
  \item [(4)] $(A,\sigma)_K$ is totally decomposable for some separable extension $K/F$.
\end{itemize}
\end{thm}

\begin{proof}
The implications $(1)\Rightarrow(2)$ and $(2)\Rightarrow(3)$ were proved in \cite[(6.2)]{dolphin3} and \cite[(8.3)]{dolphin3} respectively.
The implication $(3)\Rightarrow(4)$ follows by taking $K$ to be a separable splitting field of $A$ and $(4)\Rightarrow(1)$ follows from \cref{septd}.
\end{proof}


\footnotesize
\begin{thebibliography}{MM}
\bibitem{bayer} E. Bayer-Fluckiger, R. Parimala, A. Qu\'eguiner-Mathieu, Pfister involutions. {\it Proc. Indian Acad. Sci. Math. Sci.} {\bf 113} (2003), no. 4, 365--377.

\bibitem{becher} K. J. Becher, A proof of the Pfister factor conjecture. {\it Invent. Math.} {\bf 173} (2008), no. 1, 1--6.

\bibitem{ber}
G. Berhuy,  C. Frings,  J.-P. Tignol, Galois cohomology of the classical groups over imperfect fields.
{\it J. Pure Appl. Algebra} {\bf 211} (2007), no. 2, 307--341.

\bibitem{black} J. Black, A. Qu\'eguiner-Mathieu, Involutions, odd degree extensions and generic splitting. {\it Enseign. Math.} {\bf 60} (2014), no. 3-4, 377--395.

\bibitem{dolphin2} A. Dolphin, Decomposition of algebras with involution in characteristic $2$. {\it J. Pure Appl. Algebra} {\bf 217} (2013), no. 9, 1620--1633.

\bibitem{dolphin3} A. Dolphin, Orthogonal Pfister involutions in characteristic two. {\it J. Pure Appl. Algebra} {\bf 218} (2014), no. 10, 1900--1915.

\bibitem{kar} N. Karpenko, Hyperbolicity of orthogonal involutions. With an appendix by Jean-Pierre Tignol. {\it Doc. Math.} 2010, Extra volume: Andrei A. Suslin sixtieth birthday, 371--392.

\bibitem{kne} M. Knebusch, Grothendieck-und Wittringe von nichtausgearteten symmetrischen Bilinearformen. (German) {\it S.-B. Heidelberger Akad. Wiss. Math.-Natur. Kl.} (1969/70), 93--157.

\bibitem{knus} M.-A. Knus, A. S. Merkurjev, M. Rost, J.-P. Tignol, {\it The book of involutions}. American Mathematical Society Colloquium Publications, 44. American Mathematical Society, Providence, RI, 1998.

\bibitem{mn1} M. G. Mahmoudi, A.-H. Nokhodkar, On totally decomposable algebras with involution in characteristic two. {\it J. Algebra} {\bf 451} (2016), 208--231.

\bibitem{me} A.-H. Nokhodkar, Orthogonal involutions and totally singular
quadratic forms in characteristic two.	{\it Manuscripta Math.} doi:10.1007/s00229-017-0922-y, (2017).

\bibitem{mep} A.-H. Nokhodkar,  Pfister involutions in characteristic two. {\it Bull. London Math. Soc.} doi:10.1112/blms.12048, (2017).
\end{thebibliography}
\end{document}